\documentclass[a4paper,11pt]{amsart}

\usepackage{amsmath,amssymb,amsthm}
\usepackage{a4wide}
\usepackage{hyperref}

\usepackage{graphicx}
\usepackage{xypic}
\entrymodifiers={+!!<0pt,\fontdimen22\textfont2>}
\usepackage[all]{xy}

\newtheoremstyle{myremark} 
    {7pt}                    
    {7pt}                    
    {}  	                 
    {}                           
    {\bf}       	         
    {.}                          
    {.5em}                       
    {}  

\theoremstyle{plain}
\newtheorem*{theorem*}{Fact}
\newtheorem{lemma}{Lemma}
\newtheorem{theorem}[lemma]{Theorem}

\newtheorem{definition}[lemma]{Definition}

\newtheorem{conjecture}[lemma]{Conjecture}

\theoremstyle{myremark}

\newcommand{\htpyequiv}{\simeq}

\newcommand{\compl}[1]{\overline{#1}}

\renewcommand{\subset}{\subseteq}

\newcommand{\cl}{\mathrm{Cl}}
\newcommand{\ind}{\mathrm{Ind}}
\newcommand{\lk}{\mathrm{lk}}

\begin{document}
\title[Chordal graphs and dismantling]{A note on independence complexes of chordal graphs and dismantling}
\author{Micha{\l} Adamaszek}
\address{Department of Mathematical Sciences, University of Copenhagen, Universitetsparken 5, 2100 Copenhagen, Denmark}
\email{aszek@mimuw.edu.pl}
\thanks{Supported by the VILLUM FONDEN  network for Experimental Mathematics in Number Theory, Operator Algebras and Topology.}

\begin{abstract}
We show that the independence complex of a chordal graph is contractible if and only if this complex is dismantlable (strong collapsible) and it is homotopy equivalent to a sphere if and only if its core is a cross-polytopal sphere. The proof uses the properties of tree models of chordal graphs.
\end{abstract}

\maketitle

\section{Introduction}
\label{sect:intro}

Chordal graphs, that is graphs without an induced cycle of length more than three, form one of the fundamental classes in graph theory. They have a well known characterization in the realm of combinatorial topology:

\begin{theorem*}
\label{thm:intro}
Let $G$ be a simple, undirected graph. Then the following are equivalent:
\begin{itemize}
\item[(a)] $G$ is chordal.
\item[(b)] The clique complex of every connected induced subgraph of $G$ is contractible.
\item[(c)] The clique complex of every connected induced subgraph of $G$ is dismantlable.
\item[(d)] Every connected induced subgraph of $G$ is dismantlable.
\end{itemize}
\end{theorem*}

The purpose of this note is to show a related result in the dual case of independence complexes of chordal graphs.

\begin{theorem}
\label{thm:main}
Suppose $G$ is a chordal graph. Then the independence complex of $G$ is contractible if and only if it is dismantlable, or equivalently if $\compl{G}$ is dismantlable as a graph.
\end{theorem}

In the last section we extend this result to other homotopy types of the independence complex (Theorem~\ref{thm:extension}).

\smallskip
Independence complexes of chordal graphs have received due attention in the literature \cite{engstrom2009complexes,ehrenborg2006topology,dochtermann2009algebraic,kawamura2010independence,van2008shellable,woodroofe2009vertex,dao2013projective}. 
They are vertex-decomposable \cite{woodroofe2009vertex}, hence homotopy equivalent to wedges of spheres or to a point. Every finite wedges sum of spheres is realizable as the homotopy type of the independence complex of some chordal graph \cite{kawamura2010independence}. Moreover, every homology class in the independence complex of a chordal graph is represented by a cross-polytopal sphere corresponding to some induced matching in the graph \cite{adamaszek2012algorithmic} (called cross-cycle in \cite{jonsson2010certain}). The strong connection between topology and combinatorics for independence complexes of chordal graphs makes it plausible that also contractibility of these spaces has a combinatorial manifestation.

Dismantling is a simple operation which, if available, reduces the size of a simplicial complex without changing its homotopy type. A single dismantling step consists of removing a vertex whose link is a cone. This operation has been studied independently by many authors under the names \emph{strong collapse} \cite{barmak2012strong}, \emph{LC-reduction} \cite{civan2007linear} or \emph{link-cone reduction} \cite{matouvsek2008lc}. A sequence of dismantling steps reduces a simplicial complex to a core subcomplex whose isomorphism type does not depend on the choices made in the process \cite{barmak2012strong,matouvsek2008lc,HellNesetril2004}. A complex is dismantlable if its core is a single vertex. There is a compatible notion of dismantlability for graphs \cite{HellNesetril2004}, so that the clique complex of $G$ is dismantlable if and only if $G$ is dismantlable as a graph. Dismantlable graphs are also known as \emph{cop-win graphs} \cite{nowakowski1983vertex}.

There are other combinatorial reduction schemes of this kind: non-evasiveness and collapsibility \cite{Kozlov}. A complex is non-evasive if it is a single point or if it has a vertex whose link and deletion are both non-evasive. A complex is collapsible if it can be reduced to a single vertex by removing free faces. Dismantlability implies non-evasiveness, which implies collapsibility, and that in turn implies contractibility, but neither implication can be reversed. It is easy to observe (Lemma~\ref{lem:easy}.(d)) that a contractible independence complex of a chordal graph is non-evasive. It is the upgrade of this statement from \emph{non-evasive} to \emph{dismantlable} that requires more work. In general, dismantlability of simplicial complexes does not follow from the combination (vertex-decomposable + contractible).

On the practical side, the advantage of dismantlability over non-evasiveness or collapsibility is that it is very easy to check. As mentioned before, dismantlability can be tested greedily, by removing any available vertex whose link is a cone and continuing in the same fashion with the smaller complex. In our case this easily leads to an algorithm which solves the decision problem ``Is the independence complex of a given chordal graph $G$ contractible?'' in time $O(n^3)$, where $n$ is the number of vertices in $G$. If the complex is contractible, the algorithm provides a dismantling sequence. It may be compared with the algorithm of \cite{adamaszek2012algorithmic}, which solves the same problem in time $O(m^2)$, where $m$ is the number of edges. On the other hand, if the complex is not contractible, the algorithm of \cite{adamaszek2012algorithmic} additionally computes an induced matching representing some non-zero homology class.

\section{Prerequisites}
\label{sect:prereq}
For a simplicial complex $K$ the \emph{link} and \emph{deletion} of a vertex $v$ will be denoted $\lk_K(v)$ and $K\setminus v$, respectively. If $G$ is a simple, undirected graph then the \emph{clique complex} $\cl(G)$ and the \emph{independence complex} $\ind(G)$ are the simplicial complexes on the vertex set of $G$, whose faces are all the cliques, resp. all the independent sets of $G$. Clearly $\ind(G)=\cl(\compl{G})$, where $\compl{G}$ is the complement of $G$. The subgraph of $G$ induced by a vertex set $W\subseteq V(G)$ is denoted $G[W]$. The complete graph on $n$ vertices is denoted $K_n$. We write $N_G(u)=\{w~:~uw\in E(G)\}$ for the \emph{open neighbourhood} and $N_G[u]=N_G(u)\cup\{u\}$ for the \emph{closed neighbourhood} of a vertex $u$ in $G$. Then we have
\begin{equation}
\label{eq:links}
\ind(G)\setminus u = \ind(G\setminus u),\quad 
\lk_{\ind(G)}(u)=\ind(G\setminus N_G[u]), \quad \lk_{\cl(G)}(u)=\cl(G[N_G(u)]).
\end{equation}
A vertex $u$ of a simplicial complex $K$ is said to be \emph{dominated} by a vertex $u'$ if the link $\lk_K(u)$ is a cone with apex $u'$. The removal of a dominated vertex $u$ is called an \emph{elementary dismantling}. A complex $K$ is \emph{dismantlable} if there is a sequence of elementary dismantlings from $K$ to a point. In the language of \cite{barmak2012strong} $K$ is dismantlable if it has the strong homotopy type of a point.

\smallskip
We will now discuss the same concepts for graphs, see \cite[Sect. 2.11]{HellNesetril2004}. A vertex $u$ of $G$ is \emph{dominated} by $u'$ if and only if $N_G[u]\subseteq N_G[u']$. If $G$ can be reduced to a single vertex by successive removals of dominated vertices, then we say $G$ is \emph{dismantlable} (or \emph{cop-win}).\footnote{It should be pointed out that some authors call $u$ dominated by $u'$ when $N_G(u)\subseteq N_G(u')$. This gives a different notion of dismantling, which need not preserve the homotopy type of $\cl(G)$.} Since $N_G[u]\subseteq N_G[u']$ holds if and only if $\lk_{\cl(G)}(u)$ is a cone with apex $u'$, we immediately get the next observation.
\begin{lemma}
\label{lem:obs}
Let $G$ be a simple undirected graph.
\begin{itemize}
\item[(a)] The complex $\cl(G)$ is dismantlable if and only if $G$ is dismantlable as a graph.
\item[(b)] The complex $\ind(G)$ is dismantlable if and only if $\compl{G}$ is dismantlable as a graph.
\end{itemize}
\end{lemma}
Let us make things more explicit for independence complexes. A vertex $u$ of $\ind(G)$ is dominated by $u'$ if and only if $N_{\compl{G}}[u]\subseteq N_{\compl{G}}[u']$, and the latter is equivalent to $N_G(u')\subseteq N_G(u)$. We will give this case a name.
\begin{definition}
\label{def:good}
Let $G$ be a simple, undirected graph. A pair $(x,y)$ of distinct vertices of $G$ is called \emph{good in $G$} if $N_G(x)\subseteq N_G(y)$.
\end{definition}
Note that if $(x,y)$ is a good pair in $G$ then $xy\not\in E(G)$. As discussed before, if $(x,y)$ is a good pair in $G$ then $y$ is dominated by $x$ in $\compl{G}$ and in $\ind(G)$ and the removal of $y$ is an elementary dismantling of $\ind(G)$.\footnote{This is not the same as the notions of \emph{co-domination} and \emph{co-dismantling} defined in \cite{biyikouglu2014vertex}, since they are based on the condition $N_G[x]\subseteq N_G[y]$, and need not preserve the homotopy type of $\ind(G)$.}

\smallskip
We can now state an equivalent version of Theorem~\ref{thm:main}.
\begin{theorem}
\label{thm:main2}
If $G$ is a chordal graph such that $\ind(G)$ is contractible then either $G$ is a single vertex or $G$ has a good pair.
\end{theorem}
This statement immediately implies Theorem~\ref{thm:main} since the class of chordal graphs is hereditary, that is closed for induced subgraphs. The proof of Theorem~\ref{thm:main2} occupies the next section. In the remaining part of this section we recall various elementary properties of chordal graphs.

\smallskip
A vertex $v$ of a graph $G$ is called \emph{simplicial} if the subgraph of $G$ induced by $N_G(v)$ is a clique. For the purpose of this note a vertex $u$ of $G$ will be called \emph{peeling} if some neighbour of $u$ is simplicial.\footnote{Every \emph{peeling} vertex of $G$ is a \emph{shedding} vertex of $\ind(G)$ in the sense of vertex-decomposability. See \cite{woodroofe2009vertex}.}
It is a fundamental theorem of Dirac that every chordal graph has a simplicial vertex \cite{dirac,west2001introduction}. It follows that a chordal graph with at least one edge has a peeling vertex.

\smallskip
The next lemma is a compilation of various topological prerequisites related to independence complexes of chordal graphs. It also contains the weak analogue of Theorem~\ref{thm:main} with \emph{dismantlable} replaced by \emph{non-evasive}.
\begin{lemma}
\label{lem:easy}
Suppose $G$ is a chordal graph.
\begin{itemize}
\item[(a)] \cite{adamaszek2012splittings,engstrom2009complexes,kawamura2010independence} If $u$ is a peeling vertex then there is a homotopy equivalence
\begin{equation}
\label{eq:peel}
\ind(G)\htpyequiv \ind(G\setminus u)\vee\Sigma\ \ind(G\setminus N_G[u]).
\end{equation}
\item[(b)] \cite{woodroofe2009vertex,van2008shellable,dochtermann2009algebraic,kawamura2010independence} $\ind(G)$ is homotopy equivalent to a wedge of spheres or it is contractible.
\item[(c)] If $\ind(G)$ is contractible and $u$ is a peeling vertex then $\ind(G\setminus u)$ and $\ind(G\setminus N_G[u])$ are contractible. In particular, $V(G\setminus N_G[u])\neq\emptyset$.
\item[(d)] If $\ind(G)$ is contractible then it is non-evasive (therefore collapsible).
\item[(e)] If $G_1,\ldots,G_k$ are the connected components of $G$ then $\ind(G)$ is contractible if and only if at least one of $\ind(G_i)$ is contractible.
\item[(f)] If $(x,y)$ is a good pair in $G$ then $x$ is simplicial.
\end{itemize}
\end{lemma}
\begin{proof}
For (a) note that if $v$ is the simplicial neighbour of $u$ then $N_G[v]\subseteq N_G[u]$. This condition, together with \eqref{eq:links}, already implies \eqref{eq:peel}, since the inclusion of $\lk_{\ind(G)}(u)$ into $\ind(G)\setminus u$ factors through the contractible cone $v\ast \lk_{\ind(G)}(u)$; see \cite[Thm. 3.3]{adamaszek2012splittings} for details. Part (b) follows from (a) by induction with initial conditions $\ind(\emptyset)=S^{-1}$ and $\ind(\compl{K_n})=\bigvee^{n-1}S^0$. Part (c) follows from (a) and (b). Part (d) follows by induction from (c) and the relations \eqref{eq:links}. Part (e) follows from (b) and the fact that $\ind(G)$ is the join of $\ind(G_i)$, $i=1,\ldots,k$. Finally, for (f), note that if two vertices $z,z'\in N_G(x)$ were not adjacent, then $x,z,y,z'$ would be an induced $4$-cycle in $G$ since $xy\not\in E(G)$.
\end{proof}

\section{The main theorem}
\label{sect:proof}

In contrast with the simple argument of Lemma~\ref{lem:easy}.(d), where any peeling vertex $u$ does the job, not every simplicial vertex $x$ is part of a good pair $(x,y)$. That makes the proof of Theorem~\ref{thm:main2} more complicated. To locate good pairs we will employ the characterization of chordal graphs as intersection graphs of subtrees of a tree (see \cite{west2001introduction}).

We say that a tree $T$ (whose vertices will be called \emph{nodes}) is a \emph{tree model} of a graph $G$ if:
\begin{itemize}
\item[(a)] every node of $T$ contains a subset of vertices of $G$ (perhaps empty),
\item[(b)] for every vertex $v\in V(G)$, the nodes of $T$ which contain $v$ form a subtree of $T$,
\item[(c)] $vw\in E(G)$ if and only if there is a node of $T$ containing both $v$ and $w$.
\end{itemize}
A node of $T$ which contains a vertex $v$ will be called a \emph{$v$-node}. Now (b) states that for each vertex $v$, the $v$-nodes span a tree. A graph is chordal if and only if it has a tree model \cite{gavril1974intersection}. One common construction is  a clique tree, which has the additional property that the nodes of $T$ are in bijection with maximal cliques of $G$, but we do not make any such extra restrictions. If $R$ is any marked node of $T$ (the \emph{root}), then we call $(T,R)$ a \emph{rooted tree model} of $G$.

\begin{figure}
\label{fig:1}
\includegraphics{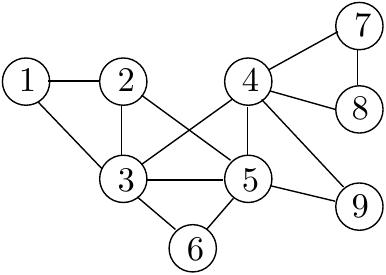} $\quad$ \includegraphics{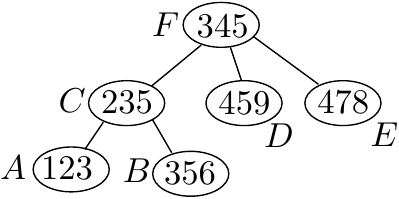} $\quad$ \includegraphics{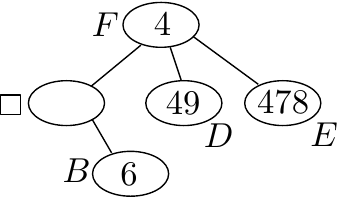}

\caption{A chordal graph $G$ with contractible $\ind(G)$, its tree model $T$ and the rooted tree model $(T/2,\square)$ of the graph $G\setminus N_G[2]=G[\{4,6,7,8,9\}]$.}
\end{figure}

If $T$ is a tree model of $G$ and $W\subset V(G)$ is any subset, then $T|_W$ will be the tree obtained from $T$ by retaining only the vertex labels that belong to $W$ and erasing those not in $W$, without changing the shape of $T$. Then $T$ is a tree model of the induced subgraph $G[W]$.

If $u\in V(G)$ is any vertex and $T$ is a tree model of $G$ then we denote by $(T/u,\square)$ the rooted tree model of $G\setminus N_G[u]$ obtained as follows: contract all the $u$-nodes of $T$ to a single node called $\square$; assign $\square$ the empty set of labels; erase all vertex labels of vertices from $N_G[u]$. The new pair $(T/u,\square)$ is clearly a rooted tree model of $G\setminus N_G[u]$.

\medskip
For the inductive arguments we will need good pairs which are compatible with a given rooted tree model of $G$ in the sense defined next.

\begin{definition}
\label{def:trgood}
Suppose $(T,R)$ is a rooted tree model of $G$. We say that  \emph{$(x,y)$ is $(T,R)$-good in $G$} if $(x,y)$ is good in $G$ and the shortest path in $T$ from any $x$-node to $R$ contains some $y$-node.
\end{definition}

Intuitively, a pair $(x,y)$ is $(T,R)$-good if the $x$-nodes are located deeper in $(T,R)$ than the $y$-nodes. The next lemma is used to promote good pairs from certain subgraphs of $G$ to good pairs in $G$.
\begin{lemma}
\label{lem:trgoodlemma}
Suppose $G$ is chordal and $u$ is any vertex of $G$.
\begin{itemize}
\item[(a)] If $x$ is an isolated vertex of $G\setminus N_G[u]$ then $(x,u)$ is a good pair in $G$.
\item[(b)] Suppose $T$ is any tree model of $G$. If $(x,y)$ is a $(T/u,\square)$-good pair in $G\setminus N_G[u]$ then $(x,y)$ is a good pair in $G$.
\item[(c)] Suppose $(T,R)$ is any rooted tree model of $G$ and $G'$ is a connected component of $G$. If $(x,y)$ is a $(T|_{V(G')},R)$-good pair in $G'$ then it is a $(T,R)$-good pair in $G$.
\end{itemize}
In particular, in either case, $x$ is a simplicial vertex of $G$ by Lemma~\ref{lem:easy}.(f).
\end{lemma}
\begin{proof}
Part (a) is obvious since the assumptions imply $N_G(x)\subseteq N_G[u]$ and $xu\not\in E(G)$.

For part (b), consider any vertex $z$ with $xz\in E(G)$. We want to show that $yz\in E(G)$. If $z\in V(G\setminus N_G[u])$ then the conclusion holds since $(x,y)$ is good in $G\setminus N_G[u]$. Next, suppose that $z\in N_G[u]$. Then there must exist two nodes $U,X$ in $T$ such that $U$ is a $u$-node, $X$ is an $x$-node and both $X,U$ are $z$-nodes. Since the set of $z$-nodes forms a subtree containing $X$ and $U$, every node on the shortest path $P$ from $X$ to $U$ in $T$ is also a $z$-node. The path $P$ induces a path $P'$ from $X$ to $\square$ in $T/u$. The path $P'$ contains a $y$-node, by the assumption that $(x,y)$ is $(T/u,\square)$-good. It follows that $P$ contains a $y$-node, and therefore $yz\in E(G)$.

Part (c) is clear since the vertices from the other components of $G$ do not affect the neighbourhoods of $x$ nor $y$.
\end{proof}

Next comes the key part of the argument, where we inductively construct good pairs with respect to arbitrary rooted tree models in \emph{connected} chordal graphs.

\begin{theorem}
\label{thm:main3}
Suppose $G$ is a connected chordal graph and $\ind(G)$ is contractible. Then
\begin{itemize}
\item[(a)] either $G$ is a single vertex,
\item[(b)] or, for any rooted tree model $(T,R)$ of $G$, there is a $(T,R)$-good pair $(x,y)$ in $G$.
\end{itemize}
\end{theorem}
\begin{proof}
We prove the theorem by induction on the number of vertices of $G$. It is obviously true when $G$ has at most one vertex.

Now suppose $G$ is any connected chordal graph with at least two vertices for which $\ind(G)$ is contractible. Fix an arbitrary rooted tree model $(T,R)$ of $G$. For any vertex $v$ in $G$ let $f(v)$ denote the minimal distance in $T$ from any $v$-node to $R$. Since $G$ is chordal and connected, it has at least one peeling vertex. Pick $u$ to be any peeling vertex which minimizes $f(u')$ among all peeling vertices $u'$ of $G$.

By Lemma~\ref{lem:easy}.(b) the complex $\ind(G\setminus N_G[u])$ is contractible, and by Lemma~\ref{lem:easy}.(e) there is a connected component $G'$ of $G\setminus N_G[u]$ such that $\ind(G')$ is contractible. We consider two cases.

\smallskip
\textit{Case 1.} $G'$ has more than one vertex. Consider the graph $G'$ together with a rooted tree model
$$\Big(\big(T/u\big)|_{V(G')},\square\Big).$$
By induction there is a pair $(x,y)$ in $G'$ which is good with respect to this model. The same pair is $(T/u,\square)$-good in $G\setminus N_G[u]$ by Lemma~\ref{lem:trgoodlemma}.(c). Now Lemma~\ref{lem:trgoodlemma}.(b) implies that the pair $(x,y)$ is good in $G$. It remains to show that it is $(T,R)$-good.

Let $X$ and $U$ be the $x$-node and the $u$-node which minimize the distance in $T$ from any $x$-node to any $u$-node. Since $xu\not\in E(G)$, these nodes are well and uniquely defined.

The assumption that $(x,y)$ is $(T/u,\square)$-good in $G\setminus N_G[u]$ implies that the shortest path in $T$ from $X$ to $U$ contains a $y$-node. Let $Y$ be the $y$-node on that path which is closest to $X$. Then $X$ and $Y$ also minimize the distance from any $x$-node to any $y$-node. See Figure~\ref{fig:3}.

Since $G'$ is connected, and therefore $N_{G'}(x)\neq\emptyset$, there exists a vertex $u'\in V(G')\subseteq V(G)$ such that $xu',yu'\in E(G')\subset E(G)$. The set of $u'$-nodes in $T$ spans a subtree, and it follows that all nodes on the shortest path from $X$ to $Y$ (inclusive) are $u'$-nodes. Note that $u'$ is peeling in $G$ since its neighbour $x$ is simplicial in $G$ by Lemma~\ref{lem:easy}.(f).

\begin{figure}
\label{fig:3}
\includegraphics[scale=1]{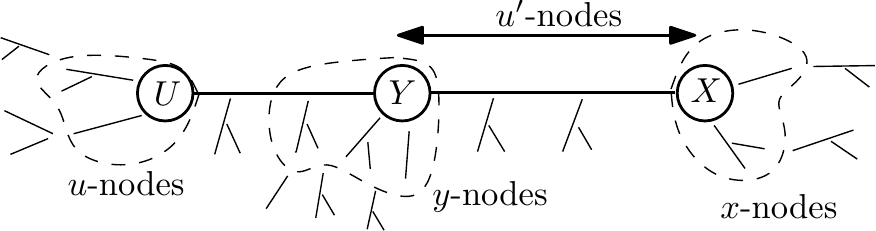}
\caption{The tree model in case 1 of the proof of Theorem~\ref{thm:main3}. All the nodes on the shortest path from $X$ to $Y$ are $u'$-nodes.}
\end{figure}

Consider all possible locations of $R$ in $T$. If $R$ and $X$ lie in the same component of $T\setminus Y$ then $f(u')<f(u)$, contradicting our choice of $u$. It follows that either $R=Y$ or $R$ and $X$ lie in different components of $T\setminus Y$. In these cases the shortest path from any $x$-node to $R$ visits $Y$, proving that $(x,y)$ is $(T,R)$-good in $G$.

\smallskip
\textit{Case 2.} $G'$ is a single vertex $x$. By Lemma~\ref{lem:trgoodlemma}.(a) the pair $(x,u)$ is good in $G$. As before, let $X$ and $U$ be the $x$-node and the $u$-node which minimize the distance in $T$ from any $x$-node to any $u$-node. Since $G$ is connected, there exists a vertex $u'\in V(G)$ such that $xu', uu'\in E(G)$, and therefore each node on the shortest path from $X$ to $U$ is a $u'$-node. Since $x$ is simplicial in $G$, its neighbour $u'$ is peeling in $G$. Now if $R$ and $X$ are in the same component of $T\setminus U$ then $f(u')<f(u)$ and we have a contradiction. Hence $R=U$ or $R$ and $U$ are in different components of $T\setminus U$, and then every path from an $x$-node to $R$ visits $U$, showing that $(x,u)$ is $(T,R)$-good in $G$.
\end{proof}

It remains to prove Theorem~\ref{thm:main2} (and thus Theorem~\ref{thm:main}).

\begin{proof}[Proof of Theorem~\ref{thm:main2}]
Let $G$ be a chordal graph with contractible $\ind(G)$ and  with at least two vertices. If $G$ is connected, then it has a good pair by Theorem~\ref{thm:main3}. If some connected component of $G$ is a single vertex $x$, then $(x,y)$ is a good pair in $G$ for any other vertex $y$. Finally, if some connected component $G'$ of $G$, with more than one vertex, satisfies that $\ind(G')$ is contractible, then the pair $(x,y)$ which is good in $G'$ is also good in $G$. By Lemma~\ref{lem:easy}.(e) that covers all the possibilities. 
\end{proof}

\section{Other results about cores}
\label{sect:remarks}

Even if $\ind(G)$ is not contractible, one can still apply  elementary dismantlings to reduce the size of the complex. The process stops when we reach a complex without a dominated vertex (such complexes are called \emph{taut} in \cite{grunbaum1970nerves} and \emph{minimal} in \cite{barmak2012strong}), which we call the \emph{core} of the original complex. It is well-defined up to isomorphism regardless of the order of elementary dismantlings, see \cite[Thm. 2.11]{barmak2012strong}, \cite{matouvsek2008lc} for complexes and \cite[Thm. 2.60]{HellNesetril2004} for graphs. Of course $\ind(G)$ is taut if and only if $G$ has no good pair. 

For example, let $F$ be a forest. It is easy to see that a forest without a good pair is either a single vertex or a disjoint union of edges. It follows that the core of $\ind(F)$ is either a point or the boundary complex of some cross-polytope. The fact that $\ind(F)$ is homotopy equivalent to a point or a sphere was first proved in \cite{ehrenborg2006topology}, and the statement about cross-polytopal cores is implicitly contained in \cite{marietti2008cores}.

The next result extends Theorem~\ref{thm:main} and generalizes the example from the previous paragraph, by considering the case when $G$ is an arbitrary chordal graph and $\ind(G)$ is homotopy equivalent to a single sphere (i.e. has total reduced Betti number $1$). Let $M_k$ denote the matching with $k$ edges, that is the graph with $2k$ vertices whose each connected component is an edge. The independence complex $\ind(M_k)$ is isomorphic to the boundary complex of the $k$-dimensional cross-polytope.

\begin{theorem}
\label{thm:extension}
If $G$ is a chordal graph with a homotopy equivalence $\ind(G)\htpyequiv S^{k-1}$ for some $k\geq 1$ then the core of $\ind(G)$ is isomorphic to $\ind(M_{k})$.
\end{theorem}
\begin{proof}
It suffices to prove the following claim: If $G$ is a chordal graph without a good pair and $\ind(G)\htpyequiv S^{k-1}$ then $G$ is isomorphic to $M_{k}$. Take any graph $G$ as in the claim and fix any tree model $T$ of $G$. 

We will first show that for any vertex $u$ of $G$ the complex $\ind(G\setminus N_G[u])$ is not contractible. Suppose otherwise, and let $G'$ be some connected component of $G\setminus N_G[u]$ such that $\ind(G')$ is contractible (it exists by Lemma~\ref{lem:easy}.(e)). If $G'$ is a single vertex $x$ then $(x,u)$ is a good pair in $G$ by Lemma~\ref{lem:trgoodlemma}.(a), a contradiction. If $G'$ has more than one vertex then, by Theorem~\ref{thm:main3}, there is a pair $(x,y)$ which is $((T/u)|_{V(G')},\square)$-good in $G'$. By Lemma~\ref{lem:trgoodlemma} parts (c) and (b) this pair is also good in $G$, which is again a contradiction. It shows that $\ind(G\setminus N_G[u])$ is not contractible.

An inductive application of Lemma~\ref{lem:easy}.(a) shows that if $v$ is a simplicial vertex of $G$ then there is a homotopy equivalence
\begin{equation}
\label{eq:long}
\ind(G)\htpyequiv \bigvee_{u\in N_G(v)} \Sigma\ \ind(G\setminus N_G[u]),
\end{equation}
see also \cite[Thm. 3.5]{engstrom2009complexes}. If $v$ has degree at least $2$ then the wedge sum in \eqref{eq:long} has at least $2$ summands, and by the previous observation each of them is homotopy equivalent to a non-contractible wedge of spheres. That contradicts the fact that $\ind(G)$ is homotopy equivalent to a single sphere. Consequently, every simplicial vertex of $G$ has degree $0$ or $1$. The first option is impossible, since then $\ind(G)$ would be contractible. It follows that every simplicial vertex of $G$ has degree $1$.

If $v$ is a vertex of degree $1$ and $y$ is any vertex in distance $2$ from $v$ then the pair $(v,y)$ is good in $G$. It means that every simplicial vertex of $G$ has no vertices at distance $2$, and it easily follows that every connected component of $G$ is a single edge, i.e. $G$ is a matching. A comparison of dimensions shows that  this matching must be $M_{k}$, which ends the proof.
\end{proof}

This yields a polynomial-time algorithm for the problem of checking if $\ind(G)$ is homotopy equivalent to a sphere for a chordal graph $G$.

\smallskip
For wedges of spheres with higher Betti numbers the situation is more complicated. For example, there are (at least) two chordal graphs whose independence complex is taut and homotopy equivalent to $S^1\vee S^1$. The first one is the disjoint union of $K_3$ and $K_2$ and the other is the $7$-vertex graph obtained from three copies $abc,def,ghi$ of $K_3$ by identifying $c$ with $d$ and $f$ with $g$. The best we can offer is the following.

\begin{conjecture}
For every finite wedge sum of spheres $X$ there are only finitely many chordal graphs $G$ such that $G$ has no good pair (i.e. $\ind(G)$ is taut) and $\ind(G)\htpyequiv X$.
\end{conjecture}
The results of this note imply the conjecture when $X$ is a point or a single sphere. We leave it as an exercise to check that it holds also when $X$ is a wedge sum of copies of $S^0$.

\bibliographystyle{plain}
\bibliography{dismantle}

\begin{thebibliography}{10}

\bibitem{adamaszek2012splittings}
Micha{\l} Adamaszek.
\newblock Splittings of independence complexes and the powers of cycles.
\newblock {\em Journal of Combinatorial Theory, Series A}, 119(5):1031--1047,
  2012.

\bibitem{adamaszek2012algorithmic}
Micha{\l} Adamaszek and Juraj Stacho.
\newblock Algorithmic complexity of finding cross-cycles in flag complexes.
\newblock In {\em Proceedings of the twenty-eighth annual symposium on
  Computational geometry}, pages 51--60. ACM, 2012.

\bibitem{barmak2012strong}
Jonathan~Ariel Barmak and Elias~Gabriel Minian.
\newblock Strong homotopy types, nerves and collapses.
\newblock {\em Discrete \& Computational Geometry}, 47(2):301--328, 2012.

\bibitem{biyikouglu2014vertex}
T{\"u}rker B{\i}y{\i}ko{\u{g}}lu and Yusuf Civan.
\newblock Vertex-decomposable graphs, codismantlability,
  {C}ohen-{M}acaulayness, and {C}astelnuovo-{M}umford regularity.
\newblock {\em The Electronic Journal of Combinatorics}, 21(1):P1--1, 2014.

\bibitem{civan2007linear}
Yusuf Civan and Erg{\"u}n Yal{\c{c}}{\i}n.
\newblock Linear colorings of simplicial complexes and collapsing.
\newblock {\em Journal of Combinatorial Theory, Series A}, 114(7):1315--1331,
  2007.

\bibitem{dao2013projective}
Hailong Dao and Jay Schweig.
\newblock Projective dimension, graph domination parameters, and independence
  complex homology.
\newblock {\em Journal of Combinatorial Theory, Series A}, 120(2):453--469,
  2013.

\bibitem{dirac}
G.A. Dirac.
\newblock On rigid circuit graphs.
\newblock {\em Abhandlungen aus dem Mathematischen Seminar der Universit\"at
  Hamburg}, 25(1--2):71--76, 1961.

\bibitem{dochtermann2009algebraic}
Anton Dochtermann and Alexander Engstr{\"o}m.
\newblock Algebraic properties of edge ideals via combinatorial topology.
\newblock {\em Electron. J. Combin}, 16(2):24, 2009.

\bibitem{ehrenborg2006topology}
Richard Ehrenborg and G{\'a}bor Hetyei.
\newblock The topology of the independence complex.
\newblock {\em European Journal of Combinatorics}, 27(6):906--923, 2006.

\bibitem{engstrom2009complexes}
Alexander Engstr{\"o}m.
\newblock Complexes of directed trees and independence complexes.
\newblock {\em Discrete Mathematics}, 309(10):3299--3309, 2009.

\bibitem{gavril1974intersection}
F\v{a}nic\v{a} Gavril.
\newblock The intersection graphs of subtrees in trees are exactly the chordal
  graphs.
\newblock {\em Journal of Combinatorial Theory, Series B}, 16(1):47--56, 1974.

\bibitem{grunbaum1970nerves}
Branko Gr{\"u}nbaum.
\newblock Nerves of simplicial complexes.
\newblock {\em Aequationes Mathematicae}, 4(1):63--73, 1970.

\bibitem{HellNesetril2004}
Pavol Hell and Jaroslav Ne\v{s}et\v{r}il.
\newblock {\em Graphs and Homomorphisms}.
\newblock Oxford University Press, 2004.

\bibitem{jonsson2010certain}
Jakob Jonsson.
\newblock Certain homology cycles of the independence complex of grids.
\newblock {\em Discrete \& Computational Geometry}, 43(4):927--950, 2010.

\bibitem{kawamura2010independence}
Kazuhiro Kawamura.
\newblock Independence complexes of chordal graphs.
\newblock {\em Discrete Mathematics}, 310(15):2204--2211, 2010.

\bibitem{Kozlov}
Dmitry~N Kozlov.
\newblock {\em Combinatorial Algebraic Topology}, volume~21 of {\em Algorithms
  and Computation in Mathematics}.
\newblock Springer, 2008.

\bibitem{marietti2008cores}
Mario Marietti and Damiano Testa.
\newblock Cores of simplicial complexes.
\newblock {\em Discrete \& Computational Geometry}, 40(3):444--468, 2008.

\bibitem{matouvsek2008lc}
Ji{\v{r}}{\'\i} Matou{\v{s}}ek.
\newblock {L}{C} reductions yield isomorphic simplicial complexes.
\newblock {\em Contributions to Discrete Mathematics}, 3(2), 2008.

\bibitem{nowakowski1983vertex}
Richard Nowakowski and Peter Winkler.
\newblock Vertex-to-vertex pursuit in a graph.
\newblock {\em Discrete Mathematics}, 43(2):235--239, 1983.

\bibitem{van2008shellable}
Adam Van~Tuyl and Rafael~H Villarreal.
\newblock Shellable graphs and sequentially {C}ohen--{M}acaulay bipartite
  graphs.
\newblock {\em Journal of Combinatorial Theory, Series A}, 115(5):799--814,
  2008.

\bibitem{west2001introduction}
Douglas~Brent West.
\newblock {\em Introduction to graph theory}, volume~2.
\newblock Prentice Hall Upper Saddle River, 2001.

\bibitem{woodroofe2009vertex}
Russ Woodroofe.
\newblock Vertex decomposable graphs and obstructions to shellability.
\newblock {\em Proceedings of the American Mathematical Society},
  137(10):3235--3246, 2009.

\end{thebibliography}

\end{document}